\documentclass[a4paper,12pt]{article}
\usepackage{times, url}
\usepackage{parskip}

\usepackage{graphicx,amsmath,amssymb,amsthm,amsfonts,bm,float,cite}
\newtheorem{theorem}{Theorem}[section]

\newtheorem{lemma}{Lemma}[section]
\newtheorem{definition}{Definition}[section]

\newtheorem{proposition}{Proposition}[section]
\newtheorem{remark}{Remark}[section]
\newtheorem{example}{Example}[section]

\usepackage[none]{hyphenat}

\usepackage[center]{caption}
\usepackage[bottom,symbol]{footmisc}
\DeclareFontEncoding{TS1}{}{}
\DeclareFontSubstitution{TS1}{cmr}{m}{n}
\DeclareTextSymbol{\tcrp}{TS1}{'251}
\DeclareTextSymbolDefault{\tcrp}{TS1}

\newcommand{\ZZ}[1]{\mathbb{Z}/#1\mathbb{Z}}
\newcommand{\card}[1]{\textbf{card}\left( #1 \right)}

\newcommand{\parties}[1]{\mathcal{P}\left( #1 \right)}

\begin{document}
\setcounter{page}{1}

\thispagestyle{empty} 

\noindent \textbf{Notes on Number Theory and Discrete
Mathematics \newline
Print ISSN 1310--5132, Online ISSN 2367--8275 \newline
XXXX, Volume XX, Number X, XXX--XXX \newline
DOI: 10.7546/nntdm.XXXX}  
\vspace{11mm}

\begin{center}
{\LARGE \bf  Distance between consecutive elements of \\[4mm] the multiplicative group of integers modulo $n$} 
\vspace{8mm}

{\Large \bf Steven Brown$^1$}
\vspace{3mm}

$^1$ 41 Boulevard du Roi, 78000 Versailles \\
e-mail: \url{steven.brown.math@gmail.com}

\end{center}

\noindent
\noindent {\bf Received:}  02 October 2023 \hfill {\bf Revised:} DD Month XXXX \\
{\bf Accepted:} DD Month XXXX  \hfill {\bf Online First:} DD Month XXXX \\[4mm] 
{\bf Abstract:}
For a prime number $p$, we consider its primorial $P:=p\#$ and $U(P):={\left(\ZZ{P}\right)}^\times$ the set of elements of the multiplicative group of integers modulo $P$ which we represent as points anticlockwise on a circle of perimeter $P$. These points considered with wrap around modulo $P$ are those not marked by the Eratosthenes sieve algorithm applied to all primes less than or equal to $p$. 

In this paper, we are mostly concerned with providing formulas to count the number of gaps of a given even length $D$ in $U(P)$ which we note $K(D,P)$. This work, presented with different notations is closely related to \cite{holt2015combinatorics}. We prove the formulas in three steps. Although only the last step relates to the problem of gaps in the Eratosthenes sieve (see section \ref{sec:gaps} page \pageref{sec:gaps} ) the previous formulas may be of interest to study occurrences of defined gaps sequences.
\begin{itemize}
\item For a positive integer $n$, we prove a general formula based on the inclusion-exclusion principle to count the number of occurrences of configurations\footnote{A defined sequence of gaps between the elements of the subset; this is referred to as a constellation in \cite{holt2015combinatorics}} in any subset of $\ZZ{n}$. (see equation \ref{eq:cccf} in theorem \ref{thm:cccf})
\item For a a square-free integer $P$, we particularize this formula when the subset of interest is $U(P)$. (see equation \ref{eq:main_formula} in theorem \ref{thm:main_formula})
\item For a prime $p$ and its primorial $P:=p\#$, we particularize the formula again to study gaps in $U(P)$. Given a positive integer $D$ representing a distance on the circle, we give formulas to count $K(D,P)$ the number of gaps of length $D$ between elements of $U(P)$. (see equation \ref{eq:kappa_for_pairs} and section \ref{sec:KDP})
\end{itemize}

In addition, we provide a formula (see equation \ref{eq:formula_MNi} in theorem \ref{thm:formula_MNi}) to count the number of occurrences of gaps of an even length $N$ that contain exactly $i$ elements of $U(P)$.

\vspace{5mm}

\section*{Acknowledgement}

I would like to thank my wife Natallia for her support and patience and to thank William Gasarch for his numerous proofreadings of the paper that greatly improved its readability. Thanks to Martin Raab for giving counterexamples to some of the conjectures that I initially made and have now been removed.
\section{Introduction}

\subsection{Notations}\label{sec:notations}

This section details some of the notations used in this paper.

\begin{itemize}

\item We denote the set of prime numbers by $\mathbb{P}$.

\item For a prime number $p$, we write $P:=p\#$ its primorial which is the product of all prime numbers less than or equal to $p$.

\item For $x \in \ZZ{n}$, we note $r(x)$ the only integer that represents $x$ in $\{0,\ldots ,n-1\}$.

\item To facilitate reading, for an integer $n$ we write $U(n)$ the set of elements of the multiplicative group of $\ZZ{n}$.
\begin{equation*}
U(n):=(\ZZ{n})^\times
\end{equation*}

\item For an integer $n$ and a set of integers $A$ we write $A/n\mathbb{Z}$ the subset of $\ZZ{n}$ composed of all the elements from $A$ taken modulo $n$.
\begin{equation*}
A/n\mathbb{Z} := \left\{ a\mod n \; \mid \; a\in A \right\}
\end{equation*}

\item For an integer $n$ and $A$ and $B$ two subsets of $\ZZ{n}$ we note
\begin{align*}
A+B := \left\{ a+b \; \mid \; a\in A \; \text{ and } b\in B \right\} \\
A-B := \left\{ a-b \; \mid \; a\in A \; \text{ and } b\in B \right\} 
\end{align*}

\item For an integer $n$, we write $\mathcal{S}(n)$ and we call support the set of all its prime divisors.
\begin{equation*}
\mathcal{S}(n) = \left\{ q \quad | \quad q \in \mathbb{P} \text{ and } q\mid n \right\}
\end{equation*}

\item With a finite set $X$ we note $\mathcal{P}(X)$ the set composed of the $2^{\card{X}}$ subsets of $X$.

\end{itemize}

\subsection{Description of the problem}

\paragraph{}
The objective of this section is to give a short introduction\footnote{This will be formalized in the following sections. }, with some examples, of the problem that we are trying to resolve. The integer $n$ is considered as being square-free and most of the time it will be the primorial of a given prime $p$. When $n$ is not square-free the results presented here may be extended with no difficulty\footnote{If $Q$ is is the largest square-free divisor of $n$, $\ZZ{n}$ can be seen as $\frac{n}{Q}$ repeated occurrences of $\ZZ{Q}$}.

\paragraph{Representation} In this paper, we work in $\ZZ{n}$ and we represent its elements counterclockwise on a circle of circumference $n$. Let $p=5$ and $P=5\#=30$ then
\begin{equation*}
U(P)=\left\{ 1,7,11,13,17,19,23,29 \right\}
\end{equation*}
We represent the elements of $U(P)$ on a circle modulo $P$ as in figure \ref{fig:u30}. We should use this example in order to illustrate the definitions to follow regarding \emph{distance} and \emph{consecutiveness}.

\begin{figure}[!h]
\centering
\includegraphics[scale=0.5]{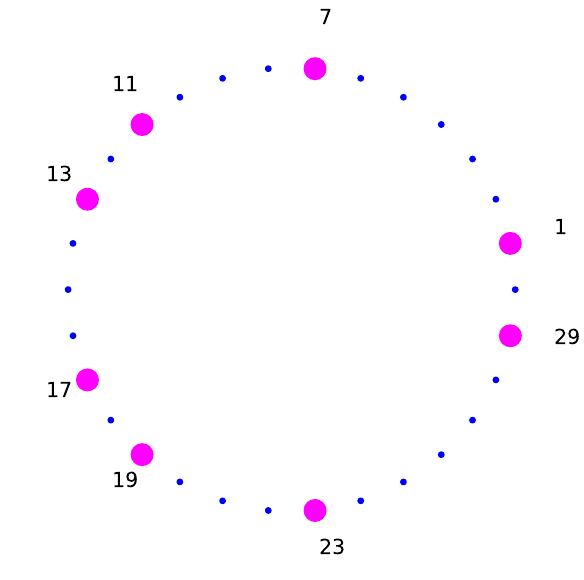}
\caption{A representation of $U(30)$}
\label{fig:u30}
\end{figure}

\paragraph{Distance on $\ZZ{n}$}
The shortest distance between two points on the circle, either clockwise or counterclockwise, noted $d$, is a distance\footnote{The proof either geometric or algebraic yields no difficulty} on $\ZZ{n}$. With the example of figure \ref{fig:u30}, we have $d(1,7)=6$ and $d(1,29)=2$.

\paragraph{Consecutiveness} For any sets $A$ and $E$ such that $A\subset E\subset \ZZ{n}$ we define a consecutiveness property consistent with the intuition: if we represent the elements of $E$ as points marked on a circle of perimeter $n$ (see for example the pink points in figure \ref{fig:u30} representing $E=U(30)$ in $\ZZ{30}$) then $A$, a subset of $E$, is deemed consecutive in $E$ if an only if there exist an arc on the circle such that:
\begin{itemize}
\item All the elements of $A$ are on that arc
\item $A$ and $E$ coincide exactly on that arc (this arc does not contain an element of $E$ which is not in $A$)
\end{itemize}

\begin{example}\label{ex:cons1} (see figure \ref{fig:u30}) The set $A=\{1,7,11\}$ is consecutive in $E$ because on the arc $\overset{\frown}{1,11}$ the set of elements of $E$ is exactly $A$.
\end{example}
\begin{example}\label{ex:cons2} (see figure \ref{fig:u30}) The set $A=\{1,7,13\}$ is not consecutive in $E$. It is obvious when we look at figure \ref{fig:u30} and here is a way to see this: an arc that contains $A$ will either contain one of these three oriented arcs $\overset{\frown}{1,13}$, $\overset{\frown}{7,1}$ or $\overset{\frown}{13,7}$. If this arc contains $\overset{\frown}{1,13}$ it contains $11$ which is in $E$ but not in $A$. If this arc contains $\overset{\frown}{7,1}$ it contains $11$, $17$, $19$, $23$ and $29$ which are in $E$ but not in $A$. If this arc contains $\overset{\frown}{13,7}$ it contains $17$, $19$, $23$ and $29$ which are in $E$ but not in $A$. Therefore there is no arc satisfying the two conditions (to contain $A$ and no element of $E\setminus A$)
\end{example}

\paragraph{Counting configurations} Now going into the subject matter, with $A\subset E\subset \ZZ{n}$ as above, we try in this paper to answer the following questions:
\begin{enumerate}
	\item \emph{Question 1} How many rotations move $A$ onto another subset included in $E$? We shall write this number $\nu(A,E,n)$. 
	\item \emph{Question 2} How many rotations move $A$ onto another subset included in $E$ that is also consecutive in $E$? We shall write this number $\kappa(A,E,n)$.
\end{enumerate}
For simplicity we can impose $A$ to contain 0 and we shall call this a \emph{configuration}.

\begin{example} (see figure \ref{fig:u30})
\begin{itemize} 
\item Answering question 1 with $A=\{0,2\}$ is being able to count the number of pairs, 2 apart, in $E$. there are 3 such pairs: $\{11,13\}$, $\{17,19\}$ and $\{29,1\}$ therefore $\nu(A,E,n)=3$.
\item Answering question 1 with $A=\{0,2,4\}$ is trying to find 3 elements in $E$ such that there exists a rotation that moves them onto $A$. One finds no suitable rotation which means that $\nu(A,E,n)=0$.
\item Answering question 2 with $A=\{0,6\}$ is trying to find pairs, 6 apart and consecutive. There are 2 such pairs $\{1,7\}$ and $\{23,29\}$. The other 4 pairs that are 6 apart are not consecutive. That means $\kappa(A,E,n)=2$ and $\nu(A,E,n)=6$
\end{itemize}
\end{example}

\subsection{Consecutiveness in a subset of $\ZZ{n}$}

\begin{definition}[Set of $E$-consecutive elements]\label{def:cons_set} Let $n$ be an integer, $A$ and $E$ two sets such that $A\subset E \subset \ZZ{n}$. We say that $A$ is a set of $E$-consecutive elements or that A is $E$-consecutive if and only if there exist two integers $x$ and $y$, $x\leq y$ such that
\begin{equation}\label{eq:consecutiveness}
A=E\bigcap \left( {[x;y[}/n\mathbb{Z} \right)
\end{equation}
In other words $A$ is exactly the intersection of $E$ and at least one interval of $\mathbb{Z}$ taken modulo $n$. When there is no ambiguity regarding the subset $E$ that we consider we shall say that $A$ is a set of consecutive elements or even that $A$ is a consecutive set. 
\end{definition}

See for instance examples \ref{ex:cons1} and \ref{ex:cons2}. In the first example take for instance $x=1$ and $y=12$. Regarding the second example there exists no suitable combination of $x$ and $y$. 

\begin{definition}[consecutive elements]\label{def:cons_elements} Let $E$ be a subset of $\ZZ{n}$. We say that two distinct elements $x$ and $y$ in $E$ are consecutive if and only if $\left\{x,y\right\}$ is an $E$-consecutive set.
\end{definition}

\begin{definition}[gap]\label{def:gap} Let $E$ be a subset of $\ZZ{n}$. When $x$ and $y$ in $E$ are consecutive we say that they define a gap of length $d(x,y)$.
\end{definition}

\begin{example} (see figure \ref{fig:u30}) 1 and 7 are consecutive whereas 1 and 11 are not.
\end{example}

%
\section{General configuration counting functions}

%
\subsection{Configuration counting functions}

In this section $n$ is an integer greater than 2 and $E$ a subset of $\ZZ{P}$. The definitions to follow are most of the time dependent on the choice of $n$ and $E$ or $n$ alone. We give the definitions mentioning this explicit dependency but, in order to simplify equations, we may not always write these arguments when the context is clear, i.e. the choice of $n$ and $E$ is clear.

\begin{definition}[Configuration $\mathcal{T}$]\label{def:configuration} We define a configuration $\mathcal{T}$ as any subset of $\ZZ{n}$ that contains 0.
\end{definition}

\begin{definition}[Configuration induced by $F$ and $x$, $\mathcal{T}(F,x)$]\label{def:configuration_induced} Let $F$ be a subset of  $\ZZ{n}$, any $x$ in $F$ defines a configuration $\mathcal{T}(F,x)$ with the subset $\mathcal{T}(F,x):=F - \left\{x\right\}$. We say that the the configuration $\mathcal{T}(F,x)$ is the configuration induced by $F$ and $x$.  
\end{definition}

\begin{example} With $n=30$ and $F=\{1,7,17\}$ we can create three configurations $\{0,6,16\}$, $\{0,10,24\}$ and $\{0,14,20\}$ by setting $x$ to respectively 1, 7 and 17.
\end{example}

\begin{definition}[Length of a configuration $L(\mathcal{T},n)$]\label{def:configuration_length} Let $n$ be an integer and $\mathcal{T}$ a  configuration of $\ZZ{n}$. We define $L(\mathcal{T},n)$ the length of the  configuration $\mathcal{T}$ as 
\begin{equation}\label{eq:configuration_length}
L(\mathcal{T},n):=\max_{t \in \mathcal{T}} r(t)
\end{equation}
With $r$ as defined in section \ref{sec:notations}. We may write $L(\mathcal{T})$ when the context is clear.
\end{definition}

\begin{definition}[Configuration core $\mathcal{C}(\mathcal{T},E,n)$ and its cardinal $\nu(\mathcal{T},E,n)$]\label{def:configuration_core} Let $n$ be an integer, $\mathcal{T}$ a configuration of $\ZZ{n}$ and $E$ a subset of $\ZZ{n}$. We define $\mathcal{C}(\mathcal{T},E,n)$, the core of the configuration $\mathcal{T}$, as the subset  of $\ZZ{n}$ composed of all the elements $x$ such that for any $t$ in $\mathcal{T}$, $x+t$ is in $E$.
\begin{equation}\label{eq:C}
\mathcal{C}(\mathcal{T},E,n):=\left\{x \in \ZZ{n} \;\mid\; \left\{x\right\}+\mathcal{T} \subset E\right\}
\end{equation}
We also define $\nu(\mathcal{T},E,n)$ as the cardinal of $\mathcal{C}(\mathcal{T},E,n)$.
\begin{equation}\label{eq:nu}
\nu(\mathcal{T},E,n):=\card{\mathcal{C}(\mathcal{T},E,n)}
\end{equation}
We may write $\mathcal{C}(\mathcal{T})$ and $\nu(\mathcal{T})$ when the context is clear. $\nu(\mathcal{T})$ counts the number of subsets $F$ of $E$ such that there exists $x\in F$ such that $\mathcal{T}=\mathcal{T}(F,x)$. In other words it counts the number of times the configuration $\mathcal{T}$ can be seen in $E$. 
\end{definition}

\begin{example} Let's take the example of figure \ref{fig:u30} with $n=30$. Let's take the configuration $\mathcal{T}=\left\{0,2,6 \right\}$ and $E=U(n)$. then $\mathcal{C}(\mathcal{T})=\left\{11,17\right\}$. $11$ is in $\mathcal{C}(\mathcal{T})$ because $\left\{11,13,17\right\}\subset U(n)$. $17$ is in $\mathcal{C}(\mathcal{T})$ because $\left\{17,19,23\right\}\subset U(n)$
\end{example}

\begin{proposition} Let $\mathcal{T}$ be a configuration and $\mathcal{C}(\mathcal{T})$ its core. We have $\mathcal{C}(\mathcal{T})\subset E$
\end{proposition}

\begin{proof} Because $\mathcal{T}$ is a configuration, it contains $0$. By definition of $\mathcal{C}(\mathcal{T})$, any $x\in \mathcal{C}(\mathcal{T})$ must in particular satisfy $x+0\in E$
\end{proof}

\begin{proposition}\label{prop:core_inclusion} Let $n$ be an integer, $E$ a subset of $\ZZ{n}$ and $\mathcal{T}_1$ and $\mathcal{T}_2$ two configurations, then
\begin{equation*}
\mathcal{T}_1 \subset \mathcal{T}_2 \quad \Rightarrow \quad \mathcal{C}(\mathcal{T}_2) \subset \mathcal{C}(\mathcal{T}_1)
\end{equation*} 
\end{proposition}

\begin{proof} Let $x\in \mathcal{C}(\mathcal{T}_2)$, then
\begin{equation*}
\forall t \in \mathcal{T}_2 \quad x+t\in E
\end{equation*}
Because $\mathcal{T}_1 \subset \mathcal{T}_2$ the above implies
\begin{equation*}
\forall t \in \mathcal{T}_1 \quad x+t\in E
\end{equation*}
which shows that $x\in \mathcal{C}(\mathcal{T}_1)$
\end{proof}

\begin{definition}[Configuration consecutive core $\mathcal{K}(\mathcal{T},E,n)$ and its cardinal $\kappa(\mathcal{T},E,n)$]\label{def:configuration_consecutive_core} Let $n$ be an integer, $E$ a subset of $\ZZ{n}$, $\mathcal{T}$ a configuration of $\ZZ{n}$. We define $\mathcal{K}(\mathcal{T},E,n)$, the consecutive core of the configuration $\mathcal{T}$, as the subset of $\ZZ{n}$ composed of all the elements $x$ such that $\left\{x\right\}+\mathcal{T}$ is $E$-consecutive. 
\begin{equation}\label{eq:K}
\mathcal{K}(\mathcal{T},E,n):=\left\{x \in \ZZ{n} \;\mid\; \left\{x\right\}+\mathcal{T} \;is\; E-consecutive \right\}
\end{equation}
We define $\kappa(\mathcal{T},E,n)$ as the cardinal of $\mathcal{K}(\mathcal{T},E,n)$.
\begin{equation}\label{eq:kappa}
\kappa(\mathcal{T},E,n):=\card{\mathcal{K}(\mathcal{T},E,n)}
\end{equation}
We may write $\mathcal{K}(\mathcal{T})$ or $\kappa(\mathcal{T})$ when the context is clear. $\kappa(\mathcal{T})$ counts the number of $E$-consecutive subsets $F$ of $E$ such that there exists $x\in F$ such that $\mathcal{T}=\mathcal{T}(F,x)$.
\end{definition}

\begin{example} Let's take the example of figure \ref{fig:u30} with $n=30$. Let's take the configuration $\mathcal{T}=\left\{0,6 \right\}$ and $E=U(n)$. then $\mathcal{K}(\mathcal{T})=\left\{1,23\right\}$. Although $7+0$ and $7+6=13$ are both in $U(n)$, 7 is not in $\mathcal{K}(\mathcal{T})$ because $\{7,13\}$ is not $U(30)$-consecutive. The same applies to $11$, $13$ and $17$ which are not in $\mathcal{K}(\mathcal{T})$ because in breach of the consecutiveness condition.
\end{example}

\begin{proposition} Let $\mathcal{T}$ be a configuration, $\mathcal{C}(\mathcal{T})$ its core and $\mathcal{K}(\mathcal{T})$ its consecutive core. We have $\mathcal{K}(\mathcal{T})\subset \mathcal{C}(\mathcal{T})$
\end{proposition}

\begin{proof} If the set $\left\{x\right\}+\mathcal{T}$ is $E$-consecutive it must satisfy $\left\{x\right\}+\mathcal{T}\subset E$
\end{proof}

\begin{definition}[Configuration complement $\Delta(\mathcal{T},n)$]\label{def:configuration_complement} Let $n$ be an integer and $\mathcal{T}$ a configuration of $\ZZ{n}$, we call configuration complement and we write $\Delta(\mathcal{T},P)$, the subset of elements $x$ in $\ZZ{n}$ which satisfy the following two conditions, $x\notin \mathcal{T}$ and $r(x)<L(\mathcal{T})$. we may write $\Delta(\mathcal{T})$ when the context is clear.
\end{definition}

\begin{remark} The configuration complement is not a configuration because it does not contain 0.
\end{remark}

\subsection{General consecutive core cardinal formula}

\begin{theorem}[Consecutive core cardinal formula]\label{thm:cccf} Let $\mathcal{T}$ be a configuration and $\Delta(\mathcal{T})$ its complement. The following equation holds true:
\begin{equation}\label{eq:cccf}
\kappa(\mathcal{T}) = \sum_{X\in \mathcal{P}(\Delta(\mathcal{T}))} (-1)^{\card{X}} \nu(\mathcal{T} \cup X)
\end{equation}
\end{theorem}

\begin{proof} $\kappa(\mathcal{T})$ counts the number of $E$-consecutive sets $F$ that contain an $x$ such that $\mathcal{T}$ is the configuration induced by $F$ and $x$. Also, $\nu(\mathcal{T})$ counts the number of sets $F\subset E$ that contain an $x$ such that $\mathcal{T}$ is the configuration induced by $F$ and $x$; irrespective of the consecutiveness condition. The idea of the proof is to say that $\kappa(\mathcal{T})$ must be equal to $\nu(\mathcal{T})$ less the number of times the set $F$ is not $E$-consecutive.

Let $\delta \in \Delta(\mathcal{T})$ and let us consider the configuration $\mathcal{T}\bigcup \left\{ \delta \right\}$. We have $\mathcal{T} \subset \mathcal{T}\bigcup \left\{ \delta \right\}$ and from proposition \ref{prop:core_inclusion} we have $\mathcal{C}(\mathcal{T}\bigcup \left\{ \delta \right\}) \subset \mathcal{C}(\mathcal{T})$. In addition $\mathcal{C}(\mathcal{T}\bigcup \left\{ \delta \right\}) \bigcap \mathcal{K}(\mathcal{T}) = \emptyset$. If $x\in \mathcal{C}(\mathcal{T}\bigcup \left\{ \delta \right\})$ then $x+\delta \in E$. By definition $\delta \notin \mathcal{T}$ and $r(\delta)<L(\mathcal{T})$ which means that $\left\{x\right\}+\mathcal{T}$ is not an $E$-consecutive set because $x+\delta$ breaches the consecutiveness condition on $\left\{x\right\} + \mathcal{T}$. Therefore $x\notin \mathcal{K}(\mathcal{T})$.

This allows us write

\begin{equation*}
\kappa(\mathcal{T}) = \nu(\mathcal{T}) - \card{ \bigcup_{\delta \in \Delta(\mathcal{T})} \mathcal{C}\left(\mathcal{T}\cup \left\{\delta \right\}\right) }
\end{equation*}

To calculate $\card{ \bigcup_{\delta \in \Delta(\mathcal{T})} \mathcal{C}\left(\mathcal{T}\cup \left\{\delta \right\}\right) }$, we use the inclusion-exclusion principle (see \cite{andrica2009basic})
\begin{equation*}
\card{ \bigcup_{\delta \in \Delta(\mathcal{T})} \mathcal{C}\left(\mathcal{T}\cup \left\{\delta \right\}\right) } = \sum_{X\in \mathcal{P}(\Delta(\mathcal{T}))\setminus \emptyset} (-1)^{\card{X}-1} \card{ \bigcap_{\delta \in X} \mathcal{C}\left(\mathcal{T}\cup \left\{\delta \right\}\right) } 
\end{equation*}
Also we have
\begin{equation*}
\bigcap_{\delta \in X} \mathcal{C}\left(\mathcal{T}\cup \left\{\delta \right\}\right) = \mathcal{C}\left(\mathcal{T}\cup X\right)
\end{equation*}
Then we can write
\begin{equation*}
\card{ \bigcap_{\delta \in X} \mathcal{C}\left(\mathcal{T}\cup \left\{\delta \right\}\right) } = \nu\left(\mathcal{T}\cup X\right)
\end{equation*}
Finally, coming back to the expression of $\kappa\left(\mathcal{T}\right)$
\begin{align*}
\kappa(\mathcal{T}) &= \nu(\mathcal{T}) - \sum_{X\in \mathcal{P}(\Delta(\mathcal{T}))\setminus \emptyset} (-1)^{\card{X}-1} \nu\left(\mathcal{T}\cup X\right) \\
&= \nu(\mathcal{T}) + \sum_{X\in \mathcal{P}(\Delta(\mathcal{T}))\setminus \emptyset} (-1)^{\card{X}} \nu\left(\mathcal{T}\cup X\right) \\
&= \sum_{X\in \mathcal{P}(\Delta(\mathcal{T}))} (-1)^{\card{X}} \nu\left(\mathcal{T}\cup X\right)
\end{align*}

\end{proof}

\begin{example} With $E:=U(30)$ as a subset of $\ZZ{30}$ as in the example of figure \ref{fig:u30}. Let $\mathcal{T}$ be the configuration $\left\{0,6\right\}$ then $\nu(\mathcal{T})=6$ because there are only 6 elements $x$ of $\ZZ{30}$ such that $x+0$ and $x+6$ are both in $E$, namely $1$, $7$, $11$, $13$, $17$ and $23$. However, $\kappa(\mathcal{T})=2$ because one should exclude the following 4 non $U(30)$-consecutive sets $\left\{7, 13\right\}$, $\left\{11, 17\right\}$, $\left\{13, 19\right\}$ and $\left\{17, 23\right\}$. There are only two $U(30)$-consecutive sets with the configuration $\mathcal{T}$ which are $\left\{1, 7\right\}$ and $\left\{23, 29\right\}$.
\end{example}

\section{A formula when $E$ is $U(P)$ for a square-free integer $P$}

In this section we consider a square-free integer $P$. We consider $\ZZ{P}$ and its subset $E:=U(P)$. Let $\mathcal{T}$ be a configuration of $\ZZ{P}$.

\begin{definition}[Configuration modulo a divisor]\label{def:configuration_modulo_divisor} Let $P$ be an integer and $\mathcal{T}$ be a configuration of $\ZZ{P}$. For any integer $q$ that divides $P$, and for any $t$ an integer modulo $P$ of $\mathcal{T}$ we define $t \mod q$ as being $r(t) \mod q$ since this number modulo $q$ is independent from the representative of $t$ that we choose.
\begin{equation*}
\forall k \in \mathbb{Z} \quad r(t) + kP \equiv r(t) \mod q
\end{equation*} 
For this reason (and only because $q$ divides $P$) we can define $\mathcal{T}/q\mathbb{Z}$ as if we were working with integers (see introduction). We define
\begin{equation*}
\mathcal{T}/q\mathbb{Z} := \left\{ r(t)\mod q \quad \mid \quad t \in \mathcal{T} \right\}
\end{equation*}
Then $\mathcal{T}/q\mathbb{Z}$ is a configuration of $\ZZ{q}$ since it is a subset of $\ZZ{q}$ that contains 0.
\end{definition}

\subsection{A formula for $\nu(\mathcal{T},U(P),P)$}

\begin{theorem}\label{thm:nu_T_UP_P} Let $P$ be a square-free integer. We consider $\ZZ{P}$ and its subset $E:=U(P)$. Let $\mathcal{T}$ be a configuration of $\ZZ{P}$.
\begin{equation}\label{eq:nu_T_UP_P}
\nu(\mathcal{T}) = \prod_{p \in \mathcal{S}(P)}\left(p-\card{ \mathcal{T}/p\mathbb{Z} } \right)
\end{equation}
\end{theorem}

\begin{proof} let $q$ be a prime \footnote{We write this prime $q$ and not $p$ to avoid confusions with $P$} that divides $P$. From \ref{def:configuration_modulo_divisor} it is possible to define $\mathcal{T}/q\mathbb{Z}$ a configuration of $\ZZ{q}$. Now according to the definition \ref{def:configuration_core} for $\mathcal{T}/q\mathbb{Z}$ in $\ZZ{q}$ with the subset $E=U(q)$ we have
\begin{align*}
\nu(\mathcal{T}/q\mathbb{Z},U(q),q) &= \card{ \left\{ x \in \ZZ{q} \; \mid \; \left\{x\right\} + \mathcal{T}/q\mathbb{Z} \subset U(q) \right\} } \\
&= \card{ \left\{ x \in \ZZ{q} \; \mid \; \forall t \in \mathcal{T}/q\mathbb{Z} \quad x\neq -t \right\} } \\
&= q - \card{ \mathcal{T}/q\mathbb{Z} } 
\end{align*}
Let's write
\begin{equation*}
X(q)=\left\{ x \in \ZZ{q} \; \mid \; \left\{x\right\} + \mathcal{T}/q\mathbb{Z} \subset U(q) \right\}
\end{equation*}

Referring to the Chinese remainder theorem (see \cite{andrica2009basic}), knowing an integer $(x \mod q)$ for all its components $q$ (all primes that divide $P$) is tantamount to knowing $(x \mod P)$. Therefore we have

\begin{equation*}
\left(x \mod P\right) \in X(P) \quad \Leftrightarrow \quad \forall q \in \mathcal{S}(P) \quad \left(x \mod q\right) \in X(q)
\end{equation*}
And therefore
\begin{align*}
\nu(\mathcal{T},U(P),P)&=\card{X(P)} \\
&= \prod_{q\in \mathcal{S}(P)}\card{X(q)} \\
&= \prod_{q\in \mathcal{S}(P)}\nu(\mathcal{T},U(q),q) \\
&= \prod_{q\in \mathcal{S}(P)} \left( q - \card{ \mathcal{T}/q\mathbb{Z} } \right) 
\end{align*}

\end{proof}

%
\subsubsection{A generalization of Euler and Nagell's totient functions}\label{sec:generalized_totient}

%
\paragraph{Euler $\phi$ function}

Euler $\phi$ function is linked to $\nu$ via

\begin{equation}\label{eq:phi_nu}
\phi(n)=\nu(\left\{0\right\},U(P),P)
\end{equation}

%
\paragraph{Nagell's totient function}

Nagell's totient function $\theta(n,P)$ counts the number of solutions of the congruence
\begin{equation*}
n\equiv x+y \mod P
\end{equation*}
under the restriction $(x,P)=(y,P)=1$ (see \cite{haukkanen1998nagell} and \cite{cohen1960nagell})

\begin{proposition}\label{prop:nu_and_Nagell}
Nagell's totient function can be expressed via $\nu$ through the following identity
\begin{equation}\label{eq:theta_nu}
\theta(n,P)=\nu(\left\{0,n\right\},U(P),P)
\end{equation}
\end{proposition}

\begin{proof} The expression $\nu(\left\{0,n\right\},U(P),P)$ counts the number of numbers $u$ in $\ZZ{P}$ such that $u+0$ and $u+n$ are both in $U(P)$. For $u$ in $\mathcal{C}(\left\{0,n\right\},U(P),P)$, we write $v=u+n$. We have $v-u=n$. Say $v=x$ and $-u=y$ and we get to $x+y=n \mod P$ and the restriction conditions are satisfied because $x$ is in $U(P)$ and is therefore coprime to $P$. And $u$ in $U(P)$ implies $y=-u$ is in $U(P)$ and therefore $(y,P)=1$. Conversly the equation $x + y \equiv n \mod P$ can be rewritten
\begin{equation*}
x - (- y) \equiv n \mod P
\end{equation*}
Because of the property $u\in U(P)\;\Leftrightarrow \; -u\in U(P)$ the above equation is simply 
\begin{equation*}
u - v \equiv n \mod P
\end{equation*}
with $x=u$ and $v=-y$

%
\paragraph{A generalized totient function}

With $\mathcal{T}$ being a configuration of $\ZZ{P}$\footnote{Although it does not contain 0, we say that $\emptyset$ is a configuration} for an integer $P$ we have seen that

\begin{itemize}
\item $P$ when $\mathcal{T}$ is $\emptyset$
\item $\phi(P)$ when $\mathcal{T}$ is $\left\{0\right\}$
\item $\theta(n,P)$ when $\mathcal{T}$ is $\left\{0,n\right\}$
\end{itemize}

For any other $\mathcal{T}$ we can consider $\nu(\mathcal{T},U(P),P)$ as some sort of generalized totient function. It seems to be a generalization of another kind than those exposed in \cite{tarnauceanu2013generalization}

\end{proof}


\subsection{Consecutive core formula when $E=U(P)$}

\begin{theorem}[main formula in $U(P)$]\label{thm:main_formula} Let $P$ be a square-free integer and $\mathcal{T}$ a configuration of $\ZZ{P}$
\begin{equation}\label{eq:main_formula}
\kappa(\mathcal{T},U(P),P) = \sum_{X\in \mathcal{P}(\Delta(\mathcal{T}))} (-1)^{\card{X}} \prod_{p\in \mathcal{S}(P)}\left(p-\card{ \left(\mathcal{T}\cup X\right)/p\mathbb{Z} } \right)
\end{equation}
\end{theorem}

\begin{proof} Application of theorems \ref{thm:cccf} and \ref{thm:nu_T_UP_P}.
\end{proof}

\subsubsection{Another expression of formula \ref{eq:main_formula}}

In this section we rewrite \ref{eq:main_formula} using the following lemma:

\begin{lemma}\label{lemma:1} Let $a$ be a positive integer and $S_a = \left\{0,2,4,\ldots ,2a-2,2a\right\}$ the set of all non negative even numbers less than or equal to $2a$. For any prime $p$ such that $p>a$ and any set $Y\subset S_a$ we have:
\begin{equation*}
\card{ Y/p\mathbb{Z} }=\card{Y}
\end{equation*}
\end{lemma}

\begin{proof} It is enough to show that the application modulo $p$ from $S_a$ to $S_a/p\mathbb{Z}$ is injective. If $x=y \mod p$ then if $x\neq y$ we must have $\left|x-y\right|= kp$ with $k\geq 2$ because $x-y$ is even. (in particular $k=1$ is excluded). That means $\left|x-y\right|\geq 2p>2a$ which contradicts $\left|x-y\right|\leq 2a$ because $x$ and $y$ are in $S_a$.
\end{proof}

\begin{proposition}\label{prop:odd} With $P=p\#$ for some prime $p$ and $\mathcal{T}$ a configuration of $\ZZ{P}$, if $\mathcal{T}$ contains an element whose representative is odd then $\nu(\mathcal{T},U(P),P)=0$
\end{proposition}

\begin{proof} It is clear that 2 divides $P$. Given that $\mathcal{T}$ contains an element which representative is odd and given definition \ref{def:configuration_modulo_divisor} which applies here we have $\mathcal{T}/2\mathbb{Z}=\left\{0,1\right\}$. From theorem \ref{thm:nu_T_UP_P}, $\nu(\mathcal{T},U(P),P)$ is a multiple of 
\begin{equation*}
2 - \card{\mathcal{T}/2\mathbb{Z}} = 2-2 = 0
\end{equation*}
\end{proof}

\begin{remark}\label{rmk:even} Because of proposition \ref{prop:odd}, when applying theorem \ref{thm:main_formula} it is sufficient to only consider configurations that only contain classes with even representatives. That is because all other terms will contribute to 0 in equation \ref{eq:main_formula} of theorem \ref{thm:main_formula} .
\end{remark}

The configuration $\mathcal{T}$ has an even length. Say $L(\mathcal{T})=D=2a$. The cardinals of the sets that belong to $\mathcal{P}(\Delta(\mathcal{T}))$ range from 0 (for $X=\emptyset$) to $a+1-\card{\mathcal{T}}$ (for $X=\Delta(\mathcal{T})$). If we split the sum over sets that have the same cardinal the equation becomes:

\begin{equation*}
\kappa(\mathcal{T},U(P),P) = \sum_{k=0}^{a+1-\card{\mathcal{T}}}\sum_{\substack{ X\in \mathcal{P}(\Delta(\mathcal{T})) \\\card{X}=k }} (-1)^{\card{X}} \prod_{p\in \mathcal{S}(P)}\left(p-\card{ \left(\mathcal{T}\cup X\right)/p\mathbb{Z} } \right)
\end{equation*}

In addition we transform the product inside the sum into the product of two products, one for primes less than or equal to $a$ and another one for primes strictly greater than $a$.

\begin{equation*}
\prod_{p\in \mathcal{S}(P)} = \prod_{\substack{p\in \mathcal{S}(P)\\p\leq a}} \prod_{\substack{p\in \mathcal{S}(P)\\a<p}}
\end{equation*}

From proposition \ref{prop:odd}, the only terms that are non zero are those where $\mathcal{T}\cup X \subset S_a$ (the representatives in $X$ are all even). When $p>a$ the application of lemma \ref{lemma:1} gives:

\begin{align*}
\card{ \left(\mathcal{T}\cup X\right)/p\mathbb{Z} } &=\card{ \mathcal{T}\cup X } \\
                                                        &= \card{\mathcal{T}}+\card{X} \\
																												&= \card{\mathcal{T}}+k
\end{align*}

The equation becomes:

\begin{multline*}
\kappa(\mathcal{T},U(P),P) = \sum_{k=0}^{a+1-\card{\mathcal{T}}} \sum_{\substack{ X\in \mathcal{P}(\Delta(\mathcal{T})) \\\card{X}=k }} (-1)^{\card{X}} \\ \prod_{\substack{p\in \mathcal{S}(P)\\p\leq a}}\left(p-\card{ \left(\mathcal{T}\cup X\right)/p\mathbb{Z} \right) }  \prod_{\substack{p\in \mathcal{S}(P)\\a<p}}(p-k-\card{\mathcal{T}})
\end{multline*}

The term $\prod_{\substack{p\in \mathcal{S}(P)\\a<p}}(p-k-\card{\mathcal{T}})$ in the inner sum does not depend on $X$. The equation can be rewritten:

\begin{multline*}
\kappa(\mathcal{T},U(P),P) = \sum_{k=0}^{a+1-\card{\mathcal{T}}} \prod_{\substack{p\in \mathcal{S}(P)\\a<p}}(p-k-\card{\mathcal{T}}) \\ \sum_{\substack{ X\in \mathcal{P}(\Delta(\mathcal{T})) \\ \card{X}=k }} (-1)^{\card{X}} \prod_{\substack{p\in \mathcal{S}(P)\\p\leq a}}\left(p-\card{ \left(\mathcal{T}\cup X\right)/p\mathbb{Z} } \right) 
\end{multline*}

Let's write:

\begin{equation}\label{eq:c}
c(a,k,P,\mathcal{T})= \sum_{\substack{ X\in \mathcal{P}(\Delta(\mathcal{T})) \\ \card{X}=k }} (-1)^{\card{X}} \prod_{\substack{p\in \mathcal{S}(P)\\p\leq a}}\left(p-\card{ \left(\mathcal{T}\cup X\right)/p\mathbb{Z} } \right)
\end{equation}

Note that because $L(\mathcal{T})=2a$ it is not necessary to write that $c$ depends on $a$ since it already depends on $\mathcal{T}$. However I prefer to write it this way for readability. The equation becomes:

\begin{equation*}
\kappa(\mathcal{T},U(P),P) = \sum_{k=0}^{a+1-\card{\mathcal{T}}} c(a,k,P,\mathcal{T}) \prod_{\substack{p\in \mathcal{S}(P)\\a<p}}(p-k-\card{\mathcal{T}}) 
\end{equation*}

Which is also

\begin{equation}\label{eq:kappa_rewritten}
\kappa(\mathcal{T},U(P),P) = \sum_{k=\card{\mathcal{T}}}^{a+1} c(a,k-\card{\mathcal{T}},P,\mathcal{T}) \prod_{\substack{p\in \mathcal{S}(P)\\a<p}}(p-k) 
\end{equation}

\subsubsection{Particularization of equation \ref{eq:kappa_rewritten} when $\mathcal{T}=\{0,2a\}$ and $P=p\#$}\label{sec:gaps}

Now, in addition to being square-free, we suppose that $P$ is the primorial of a prime $p$. For a positive integer $a$, when we set the configuration to $\mathcal{T}=\{0,2a\}$. We essentially get formulas to count the number of gaps of length $2a$. The gaps among numbers coprime to a primorial have been studied at length. (see for example \cite{ziller2020differences} (smallest even number which is not a gap), \cite{holt2015combinatorics} (asymptotic population of gaps) or \cite{iwaniec1978problem} (maximum gap))

When $\mathcal{T}=\{0,D\}$ we adopt a simpler notation:

\begin{equation}\label{eq:KDP_definition}
K(D,P) := \kappa(\mathcal{T},U(P),P)
\end{equation}

When $\mathcal{T}=\{0,2a\}$, we have $L(\mathcal{T})=2a$ and $\card{\mathcal{T}}=2$. Because $P=p\#$ the equation \ref{eq:kappa_rewritten} can be rewritten

\begin{equation}\label{eq:kappa_for_pairs}
K(D,P) = \sum_{k=2}^{a+1} c(a,k-2,P,\mathcal{T}) \prod_{\substack{q\in\mathbb{P} \\ a < q\leq p}}(q-k) 
\end{equation}

\section{Formulas to count the number of consecutive pairs in $U(P)$}

\subsection{Direct application of formula \ref{eq:kappa_for_pairs} for $D$ even and $D\leq 50$}\label{sec:KDP}

With a program, according to the equation \ref{eq:kappa_for_pairs} we calculate

\begin{equation}\label{eq:K2P}
K(2,P)=\prod_{\substack{q\in\mathbb{P} \\3\leq q\leq p}}(q-2)
\end{equation}

\begin{equation}\label{eq:K4P}
K(4,P)=\prod_{\substack{q\in\mathbb{P} \\3\leq q\leq p}}(q-2)
\end{equation}

\begin{equation}\label{eq:K6P}
K(6,P)=2\prod_{\substack{q\in\mathbb{P} \\5\leq q\leq p}} (q-2)-2\prod_{\substack{q\in\mathbb{P} \\5\leq q\leq p}} (q-3)
\end{equation}

\begin{equation}\label{eq:K8P}
K(8,P)=\prod_{\substack{q\in\mathbb{P} \\5\leq q\leq p}} (q-2)-2\prod_{\substack{q\in\mathbb{P} \\5\leq q\leq p}} (q-3)\\+\prod_{\substack{q\in\mathbb{P} \\5\leq q\leq p}} (q-4)
\end{equation}

\begin{equation}\label{eq:K10P}
K(10,P)=4\prod_{\substack{q\in\mathbb{P} \\7\leq q\leq p}} (q-2)-6\prod_{\substack{q\in\mathbb{P} \\7\leq q\leq p}} (q-3)\\+2\prod_{\substack{q\in\mathbb{P} \\7\leq q\leq p}} (q-4)
\end{equation}

\begin{equation}\label{eq:K12P}
K(12,P)=6\prod_{\substack{q\in\mathbb{P} \\7\leq q\leq p}} (q-2)-14\prod_{\substack{q\in\mathbb{P} \\7\leq q\leq p}} (q-3)\\+10\prod_{\substack{q\in\mathbb{P} \\7\leq q\leq p}} (q-4)-2\prod_{\substack{q\in\mathbb{P} \\7\leq q\leq p}} (q-5)
\end{equation}

\begin{equation}\label{eq:K14P}
K(14,P)=18\prod_{\substack{q\in\mathbb{P} \\11\leq q\leq p}} (q-2)-40\prod_{\substack{q\in\mathbb{P} \\11\leq q\leq p}} (q-3)\\+28\prod_{\substack{q\in\mathbb{P} \\11\leq q\leq p}} (q-4)-6\prod_{\substack{q\in\mathbb{P} \\11\leq q\leq p}} (q-5)
\end{equation}

We have calculated the formulas for $D$ up to 50, they are available page \pageref{KDP_listings} expressed as listings.

The application of these formulas give the numbers in table \ref{tab1}. These numbers have been checked against an exact calculation of $K(D,p\#)$

\subsection{A property on the sum of the coefficients}

\begin{proposition}\label{prop:sum_coeffs} If we write $S$ the sum of the coefficients $c(a,k-2,P,\mathcal{T})$ (see equation \ref{eq:c}) for $k$ ranging from 2 to $a+1$ as in equation \ref{eq:kappa_for_pairs} and if we denote by $q$ the greatest prime less than or equal to $a$ (or such that $2q\leq D$) we have
\begin{equation}
S := \sum_{k=2}^{a+1} c(a,k-2,P,\mathcal{T}) = K(D,q\#)
\end{equation}
That means $S$ is equal to the number of gaps of length $D$ in $U(q\#)$.
\end{proposition}

\begin{proof}
\begin{align*}
S &:= \sum_{k=2}^{a+1} c(a,k-2,P,\mathcal{T}) \\
  &= \sum_{k=0}^{a-1} c(a,k,P,\mathcal{T}) \\
  &= \sum_{k=0}^{a-1} \sum_{\substack{ X\in \mathcal{P}(\Delta(\mathcal{T})) \\ \card{X}=k }} (-1)^{\card{X}} \prod_{\substack{p\in \mathbb{P}\\p\leq a}}\left(p-\card{ \left(\mathcal{T}\cup X\right)/p\mathbb{Z} } \right) \\
  &= \sum_{\substack{ X\in \mathcal{P}(\Delta(\mathcal{T}))}} (-1)^{\card{X}} \prod_{\substack{p\in \mathbb{P}\\p\leq a}}\left(p-\card{ \left(\mathcal{T}\cup X\right)/p\mathbb{Z} } \right) \\
	&= K(D,q\#)
\end{align*}
\end{proof}

In the examples that were given ($p\leq 29$) it is interesting to note that this sum is always equal to zero. This property should not hold however for values of $p\geq 43$ as indicated in the sequence $A048670$ from The On-Line Encyclopedia of Integer Sequences  (see \cite{oeis} (the Jacobsthal function applied to primorials). Indeed in $U(43\#)$, the largest gap between consecutive elements is $90$ which is larger than $2*43=86$. (43 is the smallest prime $p$ such that the longest gap in $U(p\#)$ is greater than $2p$)

\section{Number of occurrences of gaps of length $N$ that contain exactly $i$ elements of $U(P)$}

Let $a$ be a strictly positive integer and $N:=2a$ an even integer greater than 2 representing a gap length. In this section we consider the $\nu\left(\left\{0,N\right\}\right)$ pairs of elements of $U(P)$ that are $N$ apart or equivalently the $\nu\left(\left\{0,N\right\}\right)$ gaps of length $N$. For $i$ an integer satisfying $0\leq i < a$, we provide a formula (see equation \ref{eq:formula_MNi} in theorem \ref{thm:formula_MNi}) to count $M(N,i)$ the number of occurrences of gaps of length $N$ that contain exactly $i$ elements of $U(P)$. Alternatively $M(N,i)$ is also the number of gaps of length $N$ that are composed of $i+1$ \emph{consecutive} gaps, a number that is denoted $n_{N,i+1}$ in \cite{holt2015combinatorics}.

Let us introduce the following two definitions

\begin{definition}[$X(N,i)$ and $M(N,i)$] Here, $N$ is a strictly positive even integer (Say $N:=2a$) and $i$ is an integer satisfying $0\leq i < a$. Let $E$ be the set $\left\{2,4,\ldots ,2a-2 \right\}$. We define
\begin{equation}\label{eq:def_X}
X(N,i) := \sum_{\substack{Y\in \parties{E}\\\card{Y}=i}} \nu \left(Y \cup \left\{0,N\right\} \right)
\end{equation}
and
\begin{equation}\label{eq:def_M}
M(N,i) := \sum_{\substack{Y\in \parties{E}\\\card{Y}=i}} \kappa \left(Y \cup \left\{0,N\right\} \right)
\end{equation}
\end{definition}

\begin{lemma}[Relationship between $M$ and $X$]\label{lem:rel_M_X} Let $a>0$ be a positive integer, $N:=2a$ an even integer and $i$ any integer satisfying $0\leq i < a$.
\begin{equation}\label{eq:rel_M_X}
M(N,i) = \sum_{k=0}^{a-1-i} {(-1)}^k \binom{i+k}{k} X(N,i+k)
\end{equation}
\end{lemma}

\begin{proof} $E$ is the set $\left\{2,4,\ldots ,2a-2 \right\}$ and $\card{E}=a-1$. By definition
\begin{equation*}
M(N,i) = \sum_{\substack{Y\in \parties{E}\\\card{Y}=i}} \kappa \left(Y \cup \left\{0,N\right\} \right)
\end{equation*}
Let $i$ be an integer such that $0\leq i<a$ and let $Y$ be a subset of $E$ such that $\card{Y}=i$. From theorem \ref{thm:cccf}
\begin{align*}
\kappa \left(Y \cup \left\{0,N\right\} \right) &= \sum_{Z\in \mathcal{P}(\Delta(Y \cup \left\{0,N\right\}))} (-1)^{\card{Z}} \nu \left( Y \cup \left\{0,N\right\} \cup Z \right) \\
&= \sum_{Z \in \parties{E \setminus Y}} (-1)^{\card{Z}} \nu \left( Y \cup \left\{0,N\right\} \cup Z \right)\\
&= \sum_{u=0}^{a-1-i} \sum_{\substack{Z \in \parties{E \setminus Y} \\ \card{Z}=u}} (-1)^{\card{Z}} \nu \left( Y \cup \left\{0,N\right\} \cup Z \right)
\end{align*}
Therefore:
\begin{align*}
M(N,i) &= \sum_{\substack{Y\in \parties{E}\\\card{Y}=i}} \sum_{u=0}^{a-1-i} \sum_{\substack{Z \in \parties{E \setminus Y} \\ \card{Z}=u}} (-1)^{\card{Z}} \nu \left( Y \cup \left\{0,N\right\} \cup Z \right) \\
&= \sum_{u=0}^{a-1-i} \sum_{\substack{Y\in \parties{E}\\\card{Y}=i}}  \sum_{\substack{Z \in \parties{E \setminus Y} \\ \card{Z}=u}} (-1)^{\card{Z}} \nu \left( Y \cup \left\{0,N\right\} \cup Z \right)\\
&= \sum_{u=0}^{a-1-i} \sum_{\substack{Z \in \parties{E} \\ \card{Z}=u}} \sum_{\substack{Y\in \parties{E \setminus Z}\\\card{Y}=i}}   (-1)^{\card{Z}} \nu \left( Y \cup \left\{0,N\right\} \cup Z \right)\\
&= \sum_{u=0}^{a-1-i} (-1)^u \sum_{\substack{Z \in \parties{E} \; Y \in \parties{E} \; \\ Z\cap Y=\emptyset \\\card{Z}=u \; \card{Y}=i}} \nu \left( Y \cup Z \cup \left\{0,N\right\} \right)\\
&= \sum_{u=0}^{a-1-i} (-1)^u \binom{i+u}{u} \sum_{\substack{T \in \parties{E} \\ \card{T}=i+u}} \nu \left( T \cup \left\{0,N\right\} \right)
\\
&= \sum_{u=0}^{a-1-i} (-1)^u \binom{i+u}{u} X(N,i+u)
\end{align*}
\end{proof}

\begin{theorem}[Formula for $M(N,i)$]\label{thm:formula_MNi} Let $a>0$ be a positive integer, $N:=2a$ an even integer and $i$ any integer satisfying $0\leq i < a$, we have the following expression to calculate $M(N,i)$
\begin{equation}\label{eq:formula_MNi}
M(N,i) = \sum_{k=0}^{a-1-i} {(-1)}^k \binom{i+k}{k} \sum_{\substack{Y\in \parties{E}\\\card{Y}=i+k}} \prod_{p \in \mathcal{S}(P)}\left(p-\card{ {\left( Y\cup \left\{0,N\right\} \right)}/p\mathbb{Z} } \right)
\end{equation}
\end{theorem}

\begin{proof} In equation \ref{eq:rel_M_X} from lemma \ref{lem:rel_M_X}, we replace $X(N,i)$ by its definition (see equation \ref{eq:def_X}). Then, we replace $\nu$ by its expression from formula \ref{eq:nu_T_UP_P} in theorem \ref{thm:nu_T_UP_P}.
\end{proof}

\subsection{A partition of Nagell's totient function}

\begin{proposition} Let $a>0$ be a positive integer and $N:=2a$ an even integer. Let $P:=p\#$ be the primorial of a prime $p$. With $\theta(N,P)$ being Nagell's totient function as defined in \cite{cohen1960nagell} we have the following identity
\begin{equation}\label{eq:sum_M}
\sum_{i=0}^{a-1}M(N,i)=\theta(N,P)
\end{equation}
\end{proposition}

\begin{proof} We transform the sum as follows
\begin{align*}
\sum_{i=0}^{a-1}M(N,i) &= \sum_{i=0}^{a-1} \sum_{k=0}^{a-1-i} {(-1)}^k \binom{i+k}{k} X(N,i+k) \\
&= \sum_{j=0}^{a-1} \sum_{\substack{i,k\geq 0 \\ i+k=j}} {(-1)}^k \binom{j}{k} X(N,j) \\
&= \sum_{j=0}^{a-1} X(N,j) \sum_{k=0}^j {(-1)}^k \binom{j}{k}  \\
&= \sum_{j=0}^{a-1} X(N,j) (1-1)^j \\
&= X(N,0)
\end{align*}
Also, $X(N,0)=\nu(\left\{0,N\right\})$ by definition, and $\nu(\left\{0,N\right\})=\theta(N,P)$ from proposition \ref{prop:nu_and_Nagell}.
\end{proof}

\begin{remark} This result was expected as can be seen in the alternative proof below
\end{remark}

\begin{proof} Let's define the following sets
\begin{equation*}
A(N) = \left\{ u  \; | \; u \in U(P); \; u+N \in U(P) \right\}
\end{equation*}
and
\begin{equation*}
B(N,i) = \left\{ u  \; | \; u \in U(P); \; u+N \in U(P); \; \card{U(P)\cap\left\{u+1,\ldots,u+N-1\right\}}=i \right\}
\end{equation*}
We have
\begin{equation*}
\card{A(N)}=\nu(\left\{0,N\right\})=\theta(N,P)
\end{equation*}
And
\begin{equation*}
\card{B(N,i)}=M(N,i)
\end{equation*}
Because the sets $B(N,i)$ for $i$ from 0 to $a-1$ form a partition of the set $A(N)$ we must have
\begin{equation*}
\sum_{i=0}^{a-1}M(N,i)=\theta(N,P)
\end{equation*}
\end{proof}

\section{Conclusion}

The paper proposes some formulas to calculate the number of occurrences of gap patterns in $U(P)$ based on the inclusion-exclusion principle and the Chinese Remainder Theorem. For single gaps $D=2a$ this happens to provide relatively simple formulas, at least from a theoretical perspective, to count the number of gaps in any $U(p\#)$ for any prime $p$.
Unfortunately it becomes increasingly difficult to verify the formulas given the steep growth of $p\#$. Calculating the coefficients to apply in the formulas is also very challenging since a naive calculation of $K(D,P)$ yields a complexity $2^a$. However once the coefficients are calculated for a certain even gap $D$ the calculation of $K(D,P)$ becomes possible for very large primes $p$. The formulas offer some perspectives to calculate $K(D,P)$ when it becomes impossible from a computation perspective to just count all the gaps in $U(P)$ due to the size of $P$. It may also be interesting to apply the formulas to some specific gap patterns. There are also possibilities to optimize the calculation of $K(D,P)$ which leads to interesting algorithmic questions and would enable the calculation of more of these values.

\appendix

\section{Listings for formula \ref{eq:kappa_for_pairs} for $D$ even and $6\leq D\leq 50$}\label{KDP_listings}

We give below listings for the formulas of $K(D,P)$ (see section \ref{sec:KDP} and equation \ref{eq:kappa_for_pairs}) for $D$ even ranging from 6 to 50. The first value is the smallest prime number $p^\star$ strictly greater than $a=\frac{D}{2}$. The other couples indicate a coefficient $c$ and a value $b$ such that each couple defines the contribution $c\prod_{\substack{q\in\mathbb{P} \\a<p^\star\leq q\leq p}} (q-b)$ in the formula of $K(D,P)$. In particular, the products should be taken on primes $q$ such that $a<p^\star\leq q \leq p$ and therefore the formulas are only valid for $p^\star \leq p$.

\begin{tiny}
\begin{verbatim}
D =  6:  [5, (2, 2), (-2, 3)]
D =  8:  [5, (1, 2), (-2, 3), (1, 4)]
D =  10: [7, (4, 2), (-6, 3), (2, 4)]
D =  12: [7, (6, 2), (-14, 3), (10, 4), (-2, 5)]
D =  14: [11, (18, 2), (-40, 3), (28, 4), (-6, 5)]
D =  16: [11, (15, 2), (-40, 3), (36, 4), (-12, 5), (1, 6)]
D =  18: [11, (30, 2), (-92, 3), (100, 4), (-44, 5), (6, 6)]
D =  20: [11, (20, 2), (-78, 3), (116, 4), (-80, 5), (24, 6), (-2, 7)]
D =  22: [13, (150, 2), (-504, 3), (632, 4), (-350, 5), (72, 6)]
D =  24: [13, (270, 2), (-1088, 3), (1738, 4), (-1376, 5), (540, 6), (-84, 7)]
D =  26: [17, (1620, 2), (-6688, 3), (11090, 4), (-9378, 5), (4224, 6), (-952, 7), (84, 8)]
D =  28: [17, (1782, 2), (-7400, 3), (12312, 4), (-10400, 5), (4634, 6), (-1008, 7), (80, 8)]
D =  30: [17, (3960, 2), (-19312, 3), (38958, 4), (-41768, 5), (25376, 6), (-8570, 7), (1446, 8), (-90, 9)]
D =  32: [17, (1485, 2), (-8128, 3), (18833, 4), (-23992, 5), (18255, 6), (-8428, 7), (2287, 8), (-332, 9), (20, 10)]
D =  34: [19, (23760, 2), (-122400, 3), (265734, 4), (-315120, 5), (220944, 6), (-92466, 7), (22120, 8), (-2700, 9), (128, 10)]
D =  36: [19, (44550, 2), (-248688, 3), (592204, 4), (-783298, 5), (627720, 6), (-311962, 7), (94618, 8), (-16604, 9),
          (1516, 10), (-56, 11)]
D =  38: [23, (400950, 2), (-2239104, 3), (5333232, 4), (-7045200, 5), (5612012, 6), (-2737436, 7), (788592, 8), (-120186, 9),
          (7140, 10)]
D =  40: [23, (504900, 2), (-2915840, 3), (7236810, 4), (-10062640, 5), (8559382, 6), (-4558512, 7), (1490236, 8), 
          (-279200, 9), (25632, 10), (-768, 11)]
D =  42: [23, (908820, 2), (-5777920, 3), (16006998, 4), (-25293628, 5), (25040302, 6), (-16042408, 7), (6621546, 8),
          (-1691666, 9), (243872, 10), (-16210, 11), (294, 12)]
D =  44: [23, (420750, 2), (-2834496, 3), (8400816, 4), (-14384852, 5), (15706264, 6), (-11377586, 7), (5507072, 8), 
          (-1745434, 9), (342926, 10), (-37092, 11), (1632, 12)]
D =  46: [29, (8330850, 2), (-55372800, 3), (161805900, 4), (-272787560, 5), (292530312, 6), (-207276852, 7), (97483328, 8),
          (-29693850, 9), (5502392, 10), (-541736, 11), (20016, 12)]
D =  48: [29, (15904350, 2), (-110218240, 3), (337714368, 4), (-601500212, 5), (688462352, 6), (-528267460, 7), (274911048, 8),
          (-95882496, 9), (21589178, 10), (-2909644, 11), (201728, 12), (-4972, 13)]
D =  50: [29, (10602900, 2), (-78453760, 3), (259208326, 4), (-503845272, 5), (638819850, 6), (-553348174, 7), (333200948, 8),
          (-139061136, 9), (39323050, 10), (-7178822, 11), (771934, 12), (-40444, 13), (600, 14)]
\end{verbatim}
\end{tiny}

\section{Table of values of $K(D,p\#)$}\label{KDP_table}

The values of the table below have been calculated with two different methods
\begin{itemize}
\item Just counting the consecutive pairs in $U(P)$ with a program (reference calculation)
\item Applying formula \ref{eq:kappa_for_pairs} page \pageref{eq:kappa_for_pairs}
\end{itemize}
To support the validity of equation \ref{eq:kappa_for_pairs}; it has always returned the exact number on all calculations from the table.

\begin{table}[!ht]
\centering
\caption{Values of $K(D,p\#)$, $D$ in rows, $p$ in columns}\label{tab1}
\small
\begin{tabular}{|c||c|c|c|c|c|c|c|c|c|c|}
\hline
$K(D,p\#)$ & 2 & 3 & 5 & 7 & 11 & 13 & 17 & 19 & 23 & 29 \\
\hline
\hline
2 & 1 & 1 & 3 & 15 & 135 & 1485 & 22275 & 378675 & 7952175 & 214708725 \\
\hline
4 & 0 & 1 & 3 & 15 & 135 & 1485 & 22275 & 378675 & 7952175 & 214708725 \\
\hline
6 & 0 & 0 & 2 & 14 & 142 & 1690 & 26630 & 470630 & 10169950 & 280323050 \\
\hline
8 & 0 & 0 & 0 & 2 & 28 & 394 & 6812 & 128810 & 2918020 & 83120450 \\
\hline
10 & 0 & 0 & 0 & 2 & 30 & 438 & 7734 & 148530 & 3401790 & 97648950 \\
\hline
12 & 0 & 0 & 0 & 0 & 8 & 188 & 4096 & 90124 & 2255792 & 68713708 \\
\hline
14 & 0 & 0 & 0 & 0 & 2 & 58 & 1406 & 33206 & 871318 & 27403082 \\
\hline
16 & 0 & 0 & 0 & 0 & 0 & 12 & 432 & 12372 & 362376 & 12199404 \\
\hline
18 & 0 & 0 & 0 & 0 & 0 & 8 & 376 & 12424 & 396872 & 14123368 \\
\hline
20 & 0 & 0 & 0 & 0 & 0 & 0 & 24 & 1440 & 61560 & 2594160 \\
\hline
22 & 0 & 0 & 0 & 0 & 0 & 2 & 78 & 2622 & 88614 & 3324402 \\
\hline
24 & 0 & 0 & 0 & 0 & 0 & 0 & 20 & 1136 & 48868 & 2100872 \\
\hline
26 & 0 & 0 & 0 & 0 & 0 & 0 & 2 & 142 & 7682 & 386554 \\
\hline
28 & 0 & 0 & 0 & 0 & 0 & 0 & 0 & 72 & 5664 & 324792 \\
\hline
30 & 0 & 0 & 0 & 0 & 0 & 0 & 0 & 20 & 2164 & 154220 \\
\hline
32 & 0 & 0 & 0 & 0 & 0 & 0 & 0 & 0 & 72 & 10128 \\
\hline
34 & 0 & 0 & 0 & 0 & 0 & 0 & 0 & 2 & 198 & 15942 \\
\hline
36 & 0 & 0 & 0 & 0 & 0 & 0 & 0 & 0 & 56 & 7228 \\
\hline
38 & 0 & 0 & 0 & 0 & 0 & 0 & 0 & 0 & 2 & 570 \\
\hline
40 & 0 & 0 & 0 & 0 & 0 & 0 & 0 & 0 & 12 & 1464 \\
\hline
42 & 0 & 0 & 0 & 0 & 0 & 0 & 0 & 0 & 0 & 272 \\
\hline
44 & 0 & 0 & 0 & 0 & 0 & 0 & 0 & 0 & 0 & 12 \\
\hline
46 & 0 & 0 & 0 & 0 & 0 & 0 & 0 & 0 & 0 & 2 \\
\hline
48 & 0 & 0 & 0 & 0 & 0 & 0 & 0 & 0 & 0 & 0 \\
\hline
50 & 0 & 0 & 0 & 0 & 0 & 0 & 0 & 0 & 0 & 0 \\
\hline
\end{tabular}
\end{table}
\normalsize

A direct application of the formula gave the following numbers of occurrences of gaps in $U(41\#)$ (a list of pairs between a gap length and the number of times this gap length can be seen modulo $41\#$). These values satisfy $\sum D K(D,P) = P$ however I couldn't validate them against an exact enumeration from $U(41\#)$.

\begin{tiny}
\begin{verbatim}
[[2;8499244879125],[4;8499244879125],[6;11604850743850],[8;3682730287600],[10;4396116829650],[12;3474628537016],
[14;1475437583074],[16;741616123248],[18;949982718776],[20;230780018520],[22;252605556450],[24;199070346484],
[26;47895816494],[28;45885975600],[30;31307108764],[32;3887806536],[34;4391607498],[36;3247427048],[38;606169690],
[40;756088668],[42;363563276],[44;57663276],[46;32658714],[48;29314704],[50;11018808],[52;1684756],[54;3980340],
[56;537324],[58;371574],[60;127928],[62;9262],[64;14400],[66;7680],[68;332],[70;360],[72;48],[74;2]]
\end{verbatim}
\end{tiny}

\bibliographystyle{plain}
\bibliography{biblio}

\end{document}